\documentclass{article}
\usepackage{amsmath,amssymb,amsthm,amsxtra,bm,graphicx}
\usepackage[all]{xy}
\usepackage[hypertex]{hyperref}
\newtheorem{thm}{Theorem}[section]

\newtheorem{lem}[thm]{Lemma}

\newtheorem{conj}[thm]{Conjecture}
\theoremstyle{definition}
\newtheorem{dfn}[thm]{Definition}
\newtheorem{rem}[thm]{Remark}
\newtheorem{notation}[thm]{Notation}

\newcommand{\F}{\mathbb{F}}

\newcommand{\N}{\mathbb{N}}
\newcommand{\Q}{\mathbb{Q}}
\newcommand{\Z}{\mathbb{Z}}

\newcommand{\R}{\mathbb{R}}

\newcommand{\fl}[1]{\lfloor{#1}\rfloor}

\begin{document}

\title{A note on logarithmic growth Newton polygons of $p$-adic differential equations}
\author{Shun Ohkubo
\footnote{
Graduate School of Mathematical Sciences, The University of Tokyo, 3-8-1 Komaba Meguro-ku Tokyo 153-8914, Japan. E-mail address: shuno@ms.u-tokyo.ac.jp }
}
\date{\today}

\maketitle

\begin{abstract}
In this paper, we answer a question due to Y.~Andr\'e related to B.~Dwork's conjecture on a specialization of the logarithmic growth of solutions of $p$-adic linear differential equations. Precisely speaking, we explicitly construct a $\nabla$-module $M$ over $\Q_p[[X]]_0$ of rank~$2$ such that the left endpoint of the special log-growth Newton polygon of $M$ is strictly above the left endpoint of the generic log-growth Newton polygon of $M$.
\end{abstract}

\section{Introduction}

We consider an ordinary linear $p$-adic differential equation
\[
Dy=y^{(\mu)}+f_{\mu-1}y^{(\mu-1)}+\dots+f_0y=0,
\]
where the $f_i$'s are bounded analytic functions in the unit disc $|x|<1$, with coefficients in a $p$-adic field. We assume that the differential equation has a full set of solutions $y$ that are analytic in the unit disc. For example, this assumption is satisfied for Picard-Fuchs equations.

In \cite{Dwo}, B.~Dwork studied asymptotic behavior of the solutions around the boundary $|x|=1$ and he proved that $y$ has at most logarithmic growth (log-growth) of order $\mu-1$, that is,
\[
|y|_0(r)=O((\log{(1/r)})^{1-\mu})\text{ as }r\uparrow 1,
\]
where $|\cdot|_0(r)$ means the $r$-Gaussian norm with center $0$. To obtain more precise information about the log-growth of the solutions of $Dy=0$, he defined the log-growth Newton polygon $\mathrm{NP}_{\log,0}(D)$.

Then, Dwork made the following observations and stated two fundamental conjectures on log-growth Newton polygons (see \cite[Concluding Remark~3]{Dwo2} for details): He first defined the notion of a Frobenius structure for a $p$-adic differential equation. If $Dy=0$ admits a Frobenius structure, then the solution space of $Dy=0$ is endowed with a canonical Frobenius structure. Then, the associated Frobenius Newton polygon is called the special Frobenius Newton polygon $\mathrm{NP}_{\varphi,0}(D)$. If we pull back the unit disc to the unit disc around a generic point $t$ of the unit disc, then we can obtain a $p$-adic differential equation $D_ty=0$, which is defined on the disc $|X-t|<1$. Then, we can compute the log-growth Newton polygon associated to $D_ty=0$, which is called the generic log-growth Newton polygon $\mathrm{NP}_{\log,t}(D_t)$. Moreover, $D_ty=0$ is also endowed with a Frobenius structure, and the corresponding Frobenius Newton polygon is called the generic Frobenius Newton polygon $\mathrm{NP}_{\varphi,t}(D_t)$. Then, based on a calculation for the hypergeometric differential equation with parameters $(1/2,1/2;1)$, Dwork conjectured that
%\[
%\text{Conjecture~1:special/generic log-growth Newton polygon coincides with the special/generic Frobenius Newton polygon}.
%\]
\begin{conj}\label{conj:1}
$\mathrm{NP}_{\log,0}(D)=\mathrm{NP}_{\varphi,0}(D)$ and $\mathrm{NP}_{\log,t}(D_t)=\mathrm{NP}_{\varphi,t}(D_t)$.
\end{conj}
%Conjecture~1:\mathrm{NP}_{\log,0}(D)=\mathrm{NP}_{\varphi,0}(D)\text{ and }\mathrm{NP}_{\log,t}(D_t)=\mathrm{NP}_{\varphi,t}(D_t)
%\]
\noindent Note that if Conjecture~\ref{conj:1} is true, then the special log-growth Newton polygon is above the generic log-growth Newton polygon by Grothendieck's specialization theorem for $F$-isocrystals. Thus, he also conjectured that
\begin{conj}\label{conj:2}
$\mathrm{NP}_{\log,0}(D)$ is above $\mathrm{NP}_{\log,t}(D_t)$.
\end{conj}
%\noindent without assuming the existence of a Frobenius structure on $Dy=0$.
%Since the solutions of $Dy=0$ have a log-growth without the existence of a Frobenius structure, he also conjectured that
%\begin{conj}\label{conj:2}
%$\mathrm{NP}_{\log,0}(D)$ is above $\mathrm{NP}_{\log,t}(D_t)$
%\end{conj}
%\noindent without assuming the existence of a Frobenius structure on $Dy=0$.
%\[
%\text{Conjecture~2:}\mathrm{NP}_{\log}(D)\text{ is above }\mathrm{NP}_{\log}(D_t).
%\]

Dwork stated these conjectures vaguely and even the precise formulations of the conjectures were given only recently: Conjecture~\ref{conj:1} was formulated accurately by B.~Chiarellotto and N.~Tsuzuki (\cite[Conjectures~6.8, 6.9]{CT}) and they proved their conjecture in the special cases where $\mu$ is $\le 2$ (Theorem~7.1 (2) loc. cit.), or where $D$ satisfies a certain technical condition called HPBQ (\cite[Theorem~6.5]{CT2}). The conjecture in the generic case was proved without any assumption \cite[Theorem~7.1]{CT2}. Conjecture~\ref{conj:2} was proved by Y.~Andr\'e in the following form: Since Dwork did not fix the endpoints of the log-growth Newton polygons, he first fixed the right endpoints of the log-growth Newton polygons at $(\mu,0)$. Under this convention, he proved that $\mathrm{NP}_{\log,0}(D)$ is above $\mathrm{NP}_{\log,t}(D_t)$ without assuming the existence of a Frobenius structure on $Dy=0$ (\cite[Theorem~4.1.1]{And}). Then, Andr\'e asked whether or not the left endpoint of the log-growth Newton polygon is stable under specialization. In any known example of $p$-adic differential equations at that point, the left endpoints of the special and generic Newton polygons coincide with each other. We also note that if $Dy=0$ admits a Frobenius structure and Chiarellotto-Tsuzuki's conjecture (in the special case) is true for $Dy=0$, then the left endpoints of the special and generic Newton polygons coincide with each other (\cite[Theorem~8.1]{CT2}).
%Since the endpoints of Frobenius Newton polygons are stable under specialization, 
%without the coincidence of the left endpoints

%Note that when $D$ is endowed with a Frobenius structure, the left-end points of $\mathrm{NP}(D)$ and $\mathrm{NP}(D_t)$ are expected to coincide due to Chiarellotto-Tsuzuki's conjecture (cf. later). Hence, such phenomena possibly arise only when $D$ does not admit a Frobenius structure.

The aim of this paper is to answer Andr\'e's question ``negatively'': We will explicitly construct a $\nabla$-module $M$ of rank $2$ such that the left endpoint of the special log-growth Newton polygon of $M$ is strictly above the left endpoint of the generic log-growth Newton polygon of $M$. Also, we prove that our example does not admit a Frobenius structure. Hence, the existence of our example means that equality of the endpoints is apparently a special feature of log-growth in the presence of a Frobenius structure.

%As mentioned above, our example does not admits a Frobenius structure.

%By considering a Dwork generic point $t$, which is an element of algebraically closed complete valuation field $\Omega$ containing $\Q_p$.

\subsection*{Acknowledgement}

The author thanks Professor Nobuo Tsuzuki for answering many questions on log-growth. He also pointed out a possibility of a negative answer to Andr\'e's question to the author. The author thanks the referee for detailed comments. The author is supported by Research Fellowships of Japan Society for the Promotion of Science for Young Scientists.

\section{Construction of a $p$-adic differential equation}
We first recall some notation in \cite{And}.

\begin{notation}
Let $p$ be a prime number. Let $v_p:\Q_p\to\Z\cup\{\infty\}$ be a discrete valuation such that $v_p(p)=1$. We define a norm $|\cdot|_p:\Q_p\to\R_{\ge 0}$ by $|x|_p=p^{-v_p(x)}$. For a complete valuation field $K$ with integer ring $\mathcal{O}_K$, let $K[[X]]_0:=\mathcal{O}_K[[X]][p^{-1}]$. We denote by $t$ a formal variable, i.e., a Dwork generic point and let $\Q_p\{\{t\}\}$ be the fraction field of the $p$-adic completion of $\Z_p[[t]][t^{-1}]$. For $x\in\R$, let $\fl{x}:=\max\{n\in\Z;n\le x\}$. Let $M$ be a $\nabla$-module $M$ over $\Q_p[[X]]_0$ (\cite[\S~0.3]{CT}). We denote $M_{\bar{t}}:=M\otimes_{\Q_p[[X]]_0}\Q_p\{\{t\}\}[[X-t]]_0$. We denote the special log-growth Newton polygon of $M$ by $\mathrm{NP}_{\log,\bar{0}}(M)$ (\cite[\S~3.3]{And}). We also denote the generic log-growth Newton polygon of $M$ by $\mathrm{NP}_{\log,\bar{t}}(M_{\bar{t}})$ \cite[\S~3.4]{And}.
\end{notation}

\begin{dfn}
Let $\sigma\in\R_{\ge 0}$. We define $P_{\sigma}\subset\mathbb{R}^2$ as the lower convex polygon defined by the vertices $(0,-\sigma)$, $(1,-\sigma)$, and $(2,0)$. Obviously, the slope set of $P_{\sigma}$ is $\{0,\sigma\}$.
\end{dfn}

\begin{thm}\label{thm:main}
Let $\sigma,\sigma'$ be real numbers satisfying $0\le \sigma'<\sigma<1$. Then, there exists a $\nabla$-module $M=M_{\sigma,\sigma'}$ over $\Q_p[[X]]_0$ of rank~$2$ such that
\begin{equation*}%\label{eq:main}
\mathrm{NP}_{\log,\bar{0}}(M)=P_{1-\sigma},\ \mathrm{NP}_{\log,\bar{t}}(M_{\bar{t}})=P_{1-\sigma'}.
\end{equation*}
In particular, the left endpoint of $\mathrm{NP}_{\log,\bar{0}}(M)$ is strictly above the left endpoint of $\mathrm{NP}_{\log,\bar{t}}(M_{\bar{t}})$.
\end{thm}

In the following, we will construct $M=M_{\sigma,\sigma'}$. Denote $\delta:=(\sigma-\sigma')/(1-\sigma)\in\R_{\ge 0}$. For $n\in\N$, we put
\[
a_n:=\begin{cases}
p^{\lfloor\sigma' r\rfloor}&\text{if }n=p^r(p^{\lfloor \delta r\rfloor+1}+1)-1\text{ for some }r\in\N,\\
0&\text{otherwise}
\end{cases}
\]
and
\[
f:=\sum_{n\in\N}a_nX^n=X^{p}+\dots\in\Z_p[[X]].
\]
We define $M:=\Q_p[[X]]_0e_1\oplus\Q_p[[X]]_0e_2$ endowed with an action of $d/dX$ given by
\[
\nabla\left(\frac{d}{dX}\right)
\begin{pmatrix}
e_1&e_2
\end{pmatrix}=\begin{pmatrix}
e_1&e_2
\end{pmatrix}
\begin{pmatrix}
0&-f\\
0&0
\end{pmatrix}.
\]
We put
\[
y_{s}:=\sum_{n\in\N}\frac{1}{n+1}a_{n+1}X^{n+1}\in\Q_p[[X]],
\]
\[
y_{g}:=\sum_{n\in\N}\frac{(X-t)^{n+1}}{n+1}\sum_{k\ge n}^{\infty}a_kt^{k-n}\binom{k}{n}\in\Q_p\{\{t\}\}[[X-t]].
\]
Then, the space of the horizontal sections of $M_{\Q_p[[X]]}$ and $M_{\Q_p\{\{t\}\}[[X-t]]}$ admit basis $\{e_1,y_se_1+e_2\}$ and $\{e_1,y_ge_1+e_2\}$. Therefore, to prove Theorem~\ref{thm:main}, we have only to prove that
\[
y_s\text{ is exactly of log-growth }1-\sigma,
\]
\[
y_g\text{ is exactly of log-growth }1-\sigma'.
\]

\begin{rem}
Our example $M_{\sigma,\sigma'}$ with $\sigma=\sigma'$ coincides with \cite[Example~5.3]{CT}. Also, $M_{\sigma,\sigma'}$ with $\sigma\neq\sigma'$ does not admit a Frobenius structure. In fact, if $M_{\sigma,\sigma'}$ admits a Frobenius structure, then the left endpoints of the special and generic log-growth Newton polygons coincide with each other by \cite[Theorem~7.1~(2)]{CT} and \cite[Theorem~8.1]{CT2}, which contradicts to Theorem~\ref{thm:main}.
%\begin{itemize}
%\item Our example $M_{\sigma,\sigma'}$ with $\sigma=\sigma'$ coincides with \cite[Example~5.3]{CT}.\\
%\item Our example $M=M_{\sigma,\sigma'}$
%\end{itemize}
\end{rem}

\subsection{Calculation of the log-growth of $y_s$}

We estimate $a_n/(n+1)$, which is the coefficient of $y_s$ at $X^{n+1}$, as follows. Let $\lambda\in\R_{\ge 0}$ and $n=p^r(p^{\fl{\delta r}+1}+1)-1$ for some $r\in\N$. Then, we have
\[
\left|\frac{a_n}{n+1}\right|_p/(n+1)^{\lambda}=p^{-\fl{\sigma' r}+r-(r+\fl{\delta r}+1)\lambda}/(1+p^{-\fl{\delta r}-1})^{\lambda}.
\]
When $(1-\sigma')/(1+\delta)(=1-\sigma)\le\lambda$, we have
\begin{align*}
&-\fl{\sigma' r}+r-(r+\fl{\delta r}+1)\lambda\\
\le&-\sigma' r+1+r-(r+\delta r)\lambda\\
\le&-\sigma' r+1+r-(r+\delta r)\left(\frac{1-\sigma'}{1+\delta}\right)=1.
\end{align*}
When $(1-\sigma')/(1+\delta)>\lambda$, we have
\begin{align*}
&-\fl{\sigma' r}+r-(r+\fl{\delta r}+1)\lambda\\
\ge&-\sigma' r+r-(r+\delta r+1)\lambda\\
=&r(-\sigma'+1-(1+\delta)\lambda)-\lambda\\
=&r(1+\delta)\left(\frac{1-\sigma'}{1+\delta}-\lambda\right)-\lambda,
\end{align*}
where the last term tends to $\infty$ as $n\to\infty$. Hence, we have
\[
\left|\frac{a_n}{n+1}\right|_p/(n+1)^{\lambda}
\begin{cases}
=O(1)\text{ as }n\to\infty&\text{if }1-\sigma\le\lambda\\
\to\infty\text{ as }n\to\infty&\text{if }1-\sigma>\lambda,
\end{cases}
\]
which implies that $y_s$ is exactly of log-growth $1-\sigma$.

\subsection{Calculation of the log-growth of $y_g$}
First, we prove that $y_g$ is not of log-growth $\lambda$ for any $\lambda\in [0,1-\sigma')$. We will use the following lemma:
\begin{lem}\label{lem:binom}
Let $u\in\Z_p$ and $0\le r\le s\in\N$. Then, we have
\[
\binom{p^su-1}{p^r-1}\in\Z_p^{\times}.
\]
\end{lem}
\begin{proof}
In $\F_p[[X]]/X^{p^r}\F_p[[X]]$, we have
\[
(1+X)^{p^su-1}=(1+X)^{p^su}(1+X)^{-1}=(1+X^{p^s})^u(1+X)^{-1}=(1+X)^{-1}=\sum_{i\in\N}(-X)^i.
\]
Hence, we have $\binom{p^su-1}{p^r-1}\equiv (-1)^{p^r-1}\neq 0$ in $\F_p$, which implies the assertion.
\end{proof}
When $n=p^r-1$ for some $r\in\N$, the coefficient of $y_g$ at $(X-t)^{n+1}t^{p^{r+\fl{\delta r}+1}}$ is equal to
\begin{equation}\label{eq:coeff}
\frac{1}{n+1}a_{p^r(p^{\fl{\delta r}+1}+1)-1}\binom{p^r(p^{\fl{\delta r}+1}+1)-1}{n}.
\end{equation}
By Lemma~\ref{lem:binom}, we have $\binom{p^r(p^{\fl{\delta r}+1}+1)-1}{n}\in\Z_p^{\times}$. We also have
\[
\left|\frac{a_{p^r(p^{\fl{\delta r}+1}+1)-1}}{n+1}\right|_p/(n+1)^{\lambda}=p^{-\fl{\sigma' r}+r-r\lambda}\ge p^{-\sigma' r+r-r\lambda}=p^{r(-\sigma'+1-\lambda)}.
\]
Since $-\sigma'+1-\lambda>0$ by assumption, $(\ref{eq:coeff})$ tends to $\infty$ as $n\to\infty$, which implies the assertion.

Finally, we prove that $y_g$ is of log-growth $1-\sigma'$. Put $\lambda:=1-\sigma'$. For $n\in\N$, we define
\begin{align*}
%S_1^+(n)&:=\{k\ge n;k=p^r-1\text{ for some }r\in\N_{\ge v_p(n+1)}\},\\
%S_1^-(n)&:=\{k\ge n;k=p^r-1\text{ for some }r\in\N_{<v_p(n+1)}\},\\
S^+(n)&:=\{k\ge n;k=p^r(p^{\fl{\delta r}+1}+1)-1\text{ for some }r\in\N_{\ge v_p(n+1)}\},\\
S^-(n)&:=\{k\ge n;k=p^r(p^{\fl{\delta r}+1}+1)-1\text{ for some }r\in\N_{<v_p(n+1)}\}.
\end{align*}
In the following, we fix $n\in\N$ and estimate the Gaussian norm of
\begin{equation}\label{eq:generic}
\sum_{k\in S^+(n)\coprod S^-(n)}\frac{a_k}{n+1}\binom{k}{n}t^{k-n}\in\Q_p\{\{t\}\},
%\sum_{k\in S_1^+\coprod S_1^-\coprod S_2^+\coprod S_2^-}\frac{a_k}{n+1}\binom{k}{n}t^{k-n}\in\Q_p\{\{t\}\},
\end{equation}
which is the coefficient of $y_g$ at $X^{n+1}$.

Case~$1$: $k=p^{r}(p^{\fl{\delta r}+1}+1)-1\in S^+(n)$.

We have
\[
\left|\frac{a_k}{n+1}\binom{k}{n}\right|_p/(n+1)^{\lambda}\le\left|\frac{a_k}{n+1}\right|_p/(n+1)^{\lambda}\le p^{-\fl{\sigma' r}+v_p(n+1)-v_p(n+1)\lambda}\le p^{-\sigma' r+1+\sigma'v_p(n+1)}\le p.
\]

Case~$2$: $k=p^{r}(p^{\fl{\delta r}+1}+1)-1\in S^-(n)$.

Since $(n+1)^{-1}\binom{k}{n}=(k+1)^{-1}\binom{k+1}{n+1}$, we have
\[
\left|\frac{a_k}{n+1}\binom{k}{n}\right|_p/(n+1)^{\lambda}\le \left|\frac{a_k}{k+1}\right|_p/(n+1)^{\lambda} \le p^{-\fl{\sigma' r}+r-v_p(n+1)\lambda}< p^{-\sigma' r+1+r-r\lambda}=p.
\]
Hence, $(\ref{eq:generic})=O((n+1)^{1-\sigma'})$ as $n\to\infty$, which implies that $y_g$ is of log-growth $1-\sigma'$.

\end{document}